\documentclass{article}

\usepackage[a4paper, total={7in, 8in}]{geometry}
\usepackage[utf8]{inputenc}   
\usepackage[T1]{fontenc}      
\usepackage{multicol}
\usepackage{amsmath,amsthm} 
\usepackage{amssymb,mathrsfs} 
\usepackage{amssymb}
\usepackage{amsfonts}
\usepackage[normalem]{ulem}
\usepackage{nicefrac}
\usepackage{graphicx}
\usepackage{enumerate}
\usepackage{graphicx} 
\usepackage{lmodern} 
\usepackage{mathtools}
\usepackage{dsfont}
\usepackage{physics}

\usepackage{framed}
\usepackage{color}

\newcommand\N{\mathbb{N}}
\newcommand\E{\mathbb{E}}

\newcommand\R{\mathbb{R}}

\newcommand\X{\R^d}

\newcommand\eps{\varepsilon}
\newcommand\Gc{\mathcal{G}}
\newcommand\Ac{\mathcal{A}}
\newcommand\Act{\widetilde {\mathcal{A}}}
\newcommand\Hc{\mathcal{H}}

\newcommand\proba{\mathbb{P}}
\newcommand\e{\mathrm{e}}

\newcommand\Lc{\mathcal{L}}

\newcommand\Cb{C_{\mathrm{b}}}
\newcommand\Clip{C_{\mathrm{lip}}}

\newcommand\Cinfty{C^{\infty}}

\newcommand\Dc{\mathcal{D}}
\newcommand\Xtilde{\widetilde X}

\newcommand{\vertiii}[1]{{\left\vert\kern-0.25ex\left\vert\kern-0.25ex\left\vert #1 
    \right\vert\kern-0.25ex\right\vert\kern-0.25ex\right\vert}}

\renewcommand{\leq}{\leqslant}

\renewcommand{\geq}{\geqslant}

\newtheorem{theorem}{Theorem}
\newtheorem{lemma}{Lemma}
\newtheorem{assumption}{Assumption}

\newtheorem{remark}{Remark}
\newtheorem{definition}{Definition}

\iffalse    
   \excludecomment{corrige}
\else
  
\fi

\begin{document}

\title{Stochastic viscosity approximations of Hamilton--Jacobi
  equations and variance reduction}
\author{Grégoire Ferré\footnote{Email: gregoire.ferre@ponts.org.} \\
\small Capital Fund Management, 23-25 rue de l'Université, 75007 Paris, France}

\date{\today}

\maketitle

\vspace{-0.2cm}

\vspace{0.3cm}

\abstract{
  We consider the computation of free energy-like quantities for diffusions when resorting
  to Monte Carlo simulation is necessary, for instance in high dimension.
  Such stochastic computations typically suffer from high variance, in particular in a
  low noise regime, because the expectation is dominated by rare trajectories for
  which the observable reaches large values. Although importance sampling, or
  tilting of trajectories, is now a standard technique for reducing the variance of
  such estimators, quantitative criteria for proving that a given control reduces variance
  are scarce, and often do not apply to practical situations. The goal of this work is
  to provide a quantitative criterion for assessing whether a given bias reduces variance, and
  at which scale. We rely for this on a recently introduced notion of stochastic solution
  for Hamilton--Jacobi--Bellman (HJB) equations. Based on this tool,
  we introduce the notion of $k$-stochastic viscosity approximation (SVA) of a HJB equation.
  We next prove that such approximate solutions are associated with estimators
  having a relative variance of order~$k-1$
  at log-scale. In particular, a sampling scheme built from a $1$-SVA has bounded variance
  as noise goes to zero.
  Finally, in order to show that our definition is relevant, we provide
  examples of stochastic viscosity approximations of order one and two, with a numerical
  illustration confirming our theoretical findings.
}

\section{Introduction}
\label{sec:intro}

Rare events play a crucial role in many scientific contexts, from networks
dynamics to molecular systems and climatic hazards. By rare we mean that
the frequency of appearance of such events is low compared to the impact of its
realization, relatively to the time scale of observation. If one is interested
in sampling  these phenomena to estimate probabilities of occurrence
or average impact, standard Monte Carlo methods generally lead to a high relative variance, hence
poor quality of estimation. This can be understood heuristically, as few samples
realize large values of the estimator. Hence, many samples are required to obtain
an accurate estimation of the quantity of interest. This is typically not affordable
in practice, since one single simulation of a realistic system may already be 
computationally expensive.

A natural goal is thus to reduce the relative variance of a naive
Monte Carlo estimator, and we focus here on the case of diffusions. Many
techniques have been designed so far, in particular genetic algorithms relying on
a selection mechanism~\cite{del2004feynman,angeli2018rare}, and importance
sampling, which modifies the system to enhance the probability of observing
the chosen rare event~\cite{bucklew:2004}.
We consider here the latter technique, which has been applied to various
systems of interest, see for instance~\cite{dupuis2007dynamic,lelievre2016partial,grafke-vanden-eijnden:2019} and references therein.

However, among the  possible biases that can be chosen to enhance
the probability of rare events, many are not efficient 
to achieve variance reduction -- a wrong bias may even deteriorate the
properties of the estimator~\cite{glasserman-wang:1997}. In the specific
case of diffusion processes we know that, in general, there exists
a unique optimal bias leading to a zero-variance
estimator~\cite{fleming1977exit,vanden-eijnden-weare:2012,lelievre2016partial}, and that this
control is the solution to a Hamilton--Jacobi--Bellman (HJB) partial differential equation
(PDE) depending on the parameters of the system~\cite{fleming2006controlled}. This
PDE also provides  the expectation of interest for any initial condition,
hence one does not need to resort to Monte Carlo simulation when the HJB problem is solved..

Since a close form of the optimal control is not available in most cases, it needs to
be approximated. As for any evolution PDE, finite difference techniques can be
used~\cite{crandall1984two,souganidis1985approximation}, with error estimates available
for instance in Lebesgue-like norms. More recently, a series of
works also address the approximation of Hamilton--Jacobi problems through
a more sampling-oriented point of view~\cite{vanden-eijnden-weare:2012,hartmann2012efficient,zhang2014applications,grafke-vanden-eijnden:2019}.

A major issue with all these techniques is that, although they may perform well
in some practical cases, one is generally not ensured that the approximation of the
optimal control indeed reduces variance. For the case of diffusions, it is proved for the
time consuming approach proposed 
in~\cite{vanden-eijnden-weare:2012}, under a stringent condition of locally uniform
convergence of the approximation. More recently, the authors in~\cite{guyader2020efficient}
provide general conditions for an abstract importance sampling estimator to be efficient.
Although powerful, the generality of the latter approach makes it difficult to apply
to diffusion processes.

The purpose of this paper is to provide a new criterion for assessing the
performance of an approximate control in terms of variance reduction.
For this we are inspired by the important work~\cite{leonard2021feynman}, which introduces
a notion of solution to the Feynman--Kac 
and Hamilton--Jacobi equations that relies on a stochastic interpretation of
the generator. This allows to consider non-smooth problems and solutions -- we may call
such solutions  <<stochastic solutions>>. This theory uses
a finite entropy condition in relation with the Girsanov theorem and extended generators
based on martingale techniques.
Our main contribution is to adapt this notion of solution to an approximation
context: this leads us to introduce the notion of stochastic viscosity approximation
of order~$k$ (or $k$-SVA). Such approximate controls solve the Hamilton--Jacobi equation
along  trajectories tilted by the control itself up to a small process of order~$k$
in the noise term coefficient (which we call temperature).
We  naturally introduce this condition through the Girsanov path change of measure,
and then prove that an importance sampling estimator based on such a control
is $k$-log efficient (in particular its relative variance ratio decays exponentially
at order~$k-1$), which is our main result.

In order to illustrate that our method can be applied to actual approximation
schemes, we consider the approach proposed in~\cite{grafke-vanden-eijnden:2019}
to bias the dynamics along the instanton, or reacting path of the dynamics.
We prove, under geometric conditions, that the estimator proposed in this work
is a stochastic viscosity approximation of order~$1$, hence it is $1$-log efficient.
We then turn to the next order approximation presented in~\cite{ferre2020approximate},
showing that under similar conditions it provides a stochastic viscosity approximation
of order~$2$, hence it is $2$-log efficient and with logarithmic relative variance
decaying linearly to zero. A simple numerical application confirms our theoretical findings.
We emphasize that the proposed approximation is
not a standard Taylor series expansion in the temperature
as in~\cite{fleming1992asymptotic}, but is a polynomial
expansion around a particular path called the instanton, which is quite unusual.
We also mention that this expansion for now only applies in the situation where the
instanton is unique, which is quite limiting in practical situations. Indeed, realistic
potential energy landscapes are strongly non-convex, which typically leads to
multiple instantons. We however believe that our results
can be generalized to the case of a finite number of instantons, which is the
purpose of ongoing work.

Our results are organized as follows. Section~\ref{sec:LDP} presents the general
setting of the work with our assumptions, together with reminders on importance
sampling. In Section~\ref{sec:results},
we present our main definition of stochastic viscosity approximation, and the
associated variance reduction property. The case of the approximation expanding
around the instanton path is presented in Section~\ref{sec:instantonexpand},
with an illustrative numerical application. Finally, we discuss
our results and further research directions in Section~\ref{sec:discussion}.

\section{Controlled diffusions and Hamilton--Jacobi--Bellman equations}
\label{sec:LDP}

This section is devoted to generalities on controlled diffusions and importance sampling.
We first present in Section~\ref{sec:setting} our setting together with the
basic assumptions on the dynamics, before turning in Section~\ref{sec:importance} to
the notion of importance sampling. Section~\ref{sec:optimal} recalls some elements
on the optimal control associated with our problem.

\subsection{Dynamics and free energy}
\label{sec:setting}
This section introduces some notation used throughout. 
We consider a diffusion process~$(X_t^\eps)_{t\geq 0}$ evolving in~$\X$
with $d\in\N^*$, and satisfying the following stochastic differential equation (SDE):
\begin{equation}
  \label{eq:X}
  d X_t^{\eps} = b(X_t^{\eps})\,dt + \sqrt{\eps}\sigma
  \,dB_t,
\end{equation}
where $b:\X \to \R^d$, $\sigma:\X \to \R^{d\times m}$ and~$(B_t)_{t\geq 0}$ is a
$m$-dimensional Brownian motion for some $m\in\N^*$. In the above equation,~$\eps>0$
is a parameter that may vary. At this stage the initial condition
of~$(X_t^\eps)_{t\geq 0}$ is a random variable with arbitrary probability distribution.
We denote by~$\,\cdot\,$ the
scalar product on~$\R^d$, while,~$\nabla^2$ denotes the Hessian operator and for two
matrices~$A,B$ belonging to~$\R^{d \times m}$, we write $A:B=\mathrm{Tr}(A^TB)$.
The dynamics~\eqref{eq:X} induces a semigroup of operators~$(P_t^\eps)_{t\geq 0}$ over the
space of continuous bounded functions~$\Cb(\X)$ defined by
\begin{equation}
  \label{eq:Pt}
  \forall\, \varphi\in \Cb(\X),\quad\forall\, t \geq 0,\quad
  P_t^\eps\varphi: x\in\X \to \E_x[\varphi(X_t^\eps)],
\end{equation}
where~$\E_x$ denotes expectation over all realizations of~\eqref{eq:X}
when the process is started at $x\in\X$. We will also use the
notation~$\E_{r,x}$ when the process is started at position~$x$
at time $r>0$.

The typical problem we want to address in this paper is the computation of the following
exponential-type expectation
\begin{equation}
  \label{eq:FK}
  A^\eps =  \E_{x_0}\left[ \e^{\frac1\eps
      f(X_T^\eps)}
    \right],
\end{equation}
or the associated free energy
\[
  Z^\eps = \varepsilon \log A^\eps,
\]
where~$T>0$ is a final time and $x_0\in\X$ a fixed initial condition for~\eqref{eq:X}
at time~$t=0$.  In particular~$A^\eps$ serves as a smoothed version of the probability:
\[
\proba( X_T^\eps\in S),
\]
for a Borel set $S\subset \X$. This probability, which is interesting for various
applications, coincides with~\eqref{eq:FK} when $f(x)=0$ for $x\in\X$ and $f(x) = -\infty$
for $x\notin\X$.

For simplifying technical arguments, we use rather strong assumptions gathered below,
which still encompass various interesting cases.
\begin{assumption}
  \label{as:reg}
  The following regularity requirements hold:
  \begin{enumerate}
  \item The matrix $D = \sigma \sigma^T$ is positive definite\footnote{We could also
    consider multiplicative noise under a uniform ellipticity condition.}.
  \item The field~$b$ belongs to~$\Cinfty(\X)^d$ and is globally Lipschitz, \textit{i.e.}
    there is~$\Clip$ such that
    \[
    \forall\, x,\, y\in\X,\quad |b(x) - b(y)|\leq \Clip |x-y|.
    \]
  \item The function~$f$ belongs to~$\Cinfty(\X)$ and is upper bounded.
  \end{enumerate}
\end{assumption}
Under Assumption~\ref{as:reg} it is standard that~\eqref{eq:X} has
a unique strong solution for all times~\cite{karatzas2012brownian}. The
semigroup~$(P_t^\eps)_{t\geq}$ is a strongly continuous semigoup
(by continuity of the trajectories) with elliptic regularity.
Its generator~$L^\eps$
has domain~$\Dc(L^\eps)$ and is defined by
\[
\forall\, \varphi \in \Dc(L^\eps),\quad
L^\eps \varphi = \lim_{t\to 0}\, \frac{P_t^\eps\varphi - \varphi}{t},
\]
where the limit is in supremum norm. When tested over smooth enough
functions, Itô calculus shows that~$L^\eps$ is represented by the
differential operator~$\Lc^\eps$ defined by
\begin{equation}
  \label{eq:gen}
  \forall\,\varphi\in C^2(\X),\quad 
\Lc^\eps \varphi = b\cdot\nabla\varphi + \frac{\eps}{2} D:\nabla^2\varphi.
\end{equation}
Finally, Assumption~\ref{as:reg} ensures that~$A^\eps$ is finite
for all~$\eps >0$ (because~$f$ is upper bounded).

The restrictive conditions of Assumption~\ref{as:reg} are imposed in order to avoid non-essential technical
details, but can be relaxed depending on the situation of interest. For instance,
the second point of the assumption ensures well-posedness of the SDE~\eqref{eq:X}
and the associated PDE, but could be replaced by a local Lipschitz assumption together with
Lyapunov condition on the generator~\cite{bellet2006ergodic}. Similarly, this global Lipschitz
condition is used in Section~\ref{sec:instantonexpand}
to obtain standard error estimates, but could also be replaced there by a local Lipschitz
condition together with appropriate growth conditions at infinity on the drift~$b$.
Finally, the upper boundedness of~$f$ could  be weakened under a Lyapunov
condition to allow functions that are not bounded from above, see the technique used in~\cite{ferre2020large}.

\subsection{Estimators and importance sampling}
\label{sec:importance}
In a small temperature regime and even more when the dimension~$d$ is large,
it is generally impossible to estimate the integral in~\eqref{eq:FK}
by numerical quadrature over a discretized grid. This is why we typically resort
to stochastic sampling of the expectation.
We know that the estimator of~$A^\eps$ in~\eqref{eq:FK} based on generating independent samples
of the random variable
\[
\Ac^\eps = \e^{\frac1\eps
      f(X_T^\eps)} 
\]
is unbiased but
typically has high (relative) variance. This is particularly true in the
low~$\eps$ regime, which is the one of interest in physical applications.

A possible way to reduce the variance is to control the process~\eqref{eq:X} with
a function~$g$ (for now arbitrary), and to consider the so-called tilted
process~$(\widetilde X_t^\eps)_{t\in[0,T]}$ defined through the stochastic differential
equation
\begin{equation}
  \label{eq:Xtilde}
  d \widetilde X_t^{\eps} = b(\widetilde X_t^{\eps})\,dt
  +D \nabla g(t,\widetilde X_t^{\eps})\,dt
  + \sqrt{\eps}\sigma \,d B_t.
\end{equation}
We generically denote by~$(\Xtilde_t^\eps)_{t\in[0,T]}$ a process tilted by a function~$g$
to be made precise, and belonging to the class of functions described below.
\begin{assumption}
  \label{as:novikov}
  A real-valued function~$g$ is said to be a biasing or importance sampling function if
  $g\in C^{1,2}([0,T]\times \X)$ and the Novikov condition holds, namely
  \begin{equation}
    \label{eq:novikov}
    \E_{x_0}\left[\e^{\frac12 \int_0^T |\sigma^T \nabla g|^2(\Xtilde_t^\eps)dt }
      \right] < +\infty.
  \end{equation}
\end{assumption}

Under Assumption~\ref{as:novikov}, the Girsanov change of measure
between the paths~$(X_t^\eps)_{t\in[0,T]}$ and~$(\Xtilde_t^\eps)_{t\in[0,T]}$
(see~\cite[Chapter~3, Proposition~5.12]{karatzas2012brownian} or~\cite{bellet2006ergodic})
provides
  \begin{equation}
    \label{eq:intgirs}
    A^\eps=\E_{x_0} \left[ \e^{\frac{1}{\varepsilon} f(X_T^{\varepsilon})}\right]
    = \E_{x_0} \left[ \e^{\frac{1}{\varepsilon} f(\widetilde X_T^{\varepsilon})
        - \frac{1}{2\varepsilon}
        \int_0^T |\sigma^T \nabla g|^2(t,\widetilde X_t^{\varepsilon})\, dt - \frac{1}{\sqrt{\varepsilon}} \int_0^T
        \sigma^T \nabla g(t,\widetilde X_t^{\varepsilon})\, dB_t   }\right].
  \end{equation}
  We next use It\^o formula over a trajectory of~$(\widetilde X_t^{\varepsilon})_{t\geq 0}$ with the
  generator~\eqref{eq:gen} (since~$g\in C^{1,2}([0,T]\times \X)$):
  \[
  d g(t,\widetilde X_t^{\varepsilon}) = \big( \partial_t g + \Lc^\eps g
  + \nabla g\cdot D  \nabla g\big)(t,\widetilde X_t^{\varepsilon})\,dt
  +\sqrt{\varepsilon} \sigma^T \nabla g(t,\widetilde X_t^{\varepsilon}) \,dB_t.
  \]
  Integrating in time and dividing by~$\varepsilon$, this equation becomes
  \[
  -\frac{1}{\sqrt{\varepsilon}} \int_0^T\sigma^T \nabla g(t,\widetilde X_t^{\varepsilon}) \,dB_t
  =  - \frac{g(T,\widetilde X_T^{\varepsilon}) - g(0,\widetilde X_0)}{\varepsilon}
  + \frac{1}{\varepsilon} \int_0^T \big( \partial_t g + \Lc^\eps g + |\sigma^T \nabla g|^2\big)
  (t,\widetilde X_t^{\varepsilon})\,dt.
  \]
  Inserting the above equation into~\eqref{eq:intgirs} we finally obtain
  \begin{equation}
    \label{eq:Agirsanov}
  A^\eps = \E_{x_0} \left[ \exp\left(\frac{1}{\eps}
    \big[f(\widetilde X_T^{\eps}) + g(0, x_0) - g(T,\widetilde X_T^{\eps}) \big]
    + \frac{1}{\eps} \int_0^T (\partial_t g + \Hc^\eps g) (t,\widetilde X_t^{\eps})\,dt
    \right)\right],
  \end{equation}
where we introduced the  Hamilton--Jacobi--Bellman nonlinear differential
operator~$\Hc^\eps$ defined by:
\[
  \forall\,g \in C^2 (\X),\quad
  \Hc^\eps g  = b\cdot\nabla g +\frac \eps 2 D:\nabla^2 g + \frac 12 |\sigma^T \nabla g|^2.
\]
One can easily check that~$\Hc^\eps$ is actually the logarithmic transform
of~$\Lc^\eps$, namely
\begin{equation}
  \label{eq:HLexpo}
\forall\,g \in C^2(\X),\quad
\Hc^\eps g =\eps\, \e^{ - g/\eps } \Lc^\eps \e^{ g/\eps}.
\end{equation}
This logarithmic transform is a standard tool for studying partial differential
equations related to control problems,
see for instance~\cite{fleming2006controlled,pham2009continuous,evans2010partial}.

In what follows, for a function~$g$ to be specified and satisfying
Assumption~\ref{as:novikov}, we denote by
\[
  \Gc^\eps =  \exp\left(\frac{1}{\eps}
    \big[ g(0, x_0) - g (T,\widetilde X_T^{\eps}) \big]
    + \frac{1}{\eps} \int_0^T (\partial_t g + \Hc^\eps g)(t,\widetilde X_t^{\eps})\,dt
    \right)
\]
the Girsanov weight of the path change of measure. Therefore~\eqref{eq:Agirsanov}
rewrites:
\[
A^\eps = \E_{x_0}\left[ \e^{\frac1\eps
      f(\Xtilde_T^\eps)} \Gc^\eps
    \right].
\]
As a consequence of this computation, the estimator of~$A^\eps$
based on independent realizations of the random variable
\begin{equation}
  \label{eq:Acteps}
\Act^\eps = \e^{\frac1\eps f(\Xtilde_T^\eps) } \Gc^\eps
\end{equation}
provides an unbiased estimator of~$A^\eps$, \textit{i.e.} it holds
$\E_{x_0}[\Act^\eps]=\E_{x_0}[\Ac^\eps]=A^\eps$.

The goal of importance sampling is to find the best
biasing function~$g$ to minimize the relative variance of the tilted estimator~$\Act^\eps$.
In order to provide a criterion for the search of the best tilting~$g$, we recall that
the relative variance of the tilted estimator (see for
instance~\cite{vanden-eijnden-weare:2012} and references therein) is defined by
\[
\frac{\E_{x_0}\left[ \left(\e^{\frac 1\eps f(\Xtilde_T^\eps) }\Gc^\eps\right)^2
\right] -\E_{x_0}\left[\e^{\frac1 \eps f(\Xtilde_T^\eps) }\Gc^\eps
\right]^2 }
{\E_{x_0}\left[\e^{\frac 1 \eps f(X_T^\eps)} \right]^2}
=
\frac{\E_{x_0}\left[ \e^{\frac 2 \eps f(\Xtilde_T^\eps) }(\Gc^\eps)^2
\right]}
{\E_{x_0}\left[\e^{\frac 1 \eps f(X_T^\eps)} \right]^2}
- 1.
\]
We thus consider the ratio:
\[
\rho(\eps)= \frac{\E_{x_0}\left[ \e^{\frac 2 \eps f(\Xtilde_T^\eps) }(\Gc^\eps)^2
  \right]}
    {\E_{x_0}\left[\e^{\frac 1 \eps f(X_T^\eps)} \right]^2},
    \]
which controls the evolution of the relative variance as~$\eps\to 0$.
Since~$\rho$ typically grows exponentially as $\eps\to 0$, we are interested
in the evolution of~$\rho$ at log-scale, and we thus introduce the following standard
definition~\cite{vanden-eijnden-weare:2012}.
\begin{definition}
\label{def:logef}
The relative log efficiency of an estimator is defined by
\begin{equation}
  \label{eq:Rcoeff}
R(\varepsilon) = \eps \log \rho(\eps).
\end{equation}
We say that an estimator is $k$-log efficient if there is $k\geq 0$ such that
\[
R(\varepsilon) = \mathrm{O}(\varepsilon^k)
\]
when $\eps\to 0$.
\end{definition}
In words,  $k$-log efficiency describes a class of estimators
with a small relative variance at exponential scale, in the low temperature regime.
Although the notation~$k$ suggests that~$k$ is an integer (which will be the case in our
examples), this does not need to be the case in general. Note also
that, when~$0<k<1$, the relative log efficiency goes to zero but the relative
variance might not even bounded (in the case where~$R(\eps)/\eps\to+\infty$),
it is only bounded by a term growing at a subexponential rate.
We are thus mostly interested in the case~$k\geq 1$. In this situation, there
exists a constant~$C>0$ such that
\[
0\leq \rho(\eps) - 1\leq C\eps^{k-1},
\]
hence the relative variance is under control.
As a result, from Definition~\ref{def:logef}, for a given small~$\eps$,
an estimator with a larger~$k$ is likely to have a variance lower by one order
of magnitude. This is what we observe in the numerical simulations performed
in Section~\ref{sec:numerical}. Note however that reducing variance
by improving log efficiency should not be done at the cost of too much
additional computational time.

\subsection{The optimal control}
\label{sec:optimal}
There is actually a simple way to find a solution to the problem of minimizing the relative variance
of~$\Act^\eps$. Consider for this the Girsanov formula~\eqref{eq:Agirsanov}: all the
terms are stochastic except~$g(0,x_0)$, which is deterministic. It is thus natural
to set~$g=g^\eps$ solution to
\begin{equation}
  \label{eq:HJB}
  \left\{
  \begin{aligned}
    (\partial_t g^{\eps}  + \Hc^\eps g^{\eps})(t,x)  &= 0, & \forall\,(t,x)\in [0,T)\times \X, \\
      g^{\eps}(T,x) & = f(x) , & \forall\,x\in\X.
  \end{aligned}
  \right.
\end{equation}
If we assume that the solution to the above equation is
well-defined, having~$(\Xtilde_t^\eps)_{t\in[0,T]}$ biased with~$g^\eps$ indeed makes the
random variable~$\Act^\eps$ deterministic, hence the associated estimator has zero variance.
Actually, under Assumption~\ref{as:reg}, standard parabolic estimates
show that there is a unique smooth solution~$g^\eps$ to~\eqref{eq:HJB} for all $\eps >0$,
    which solves the problem of variance reduction.

Moreover, inserting~$g^\eps$ into~\eqref{eq:Agirsanov}, we see that
\[
 A^\eps=\e^{\frac{g^\eps(0,x_0)}{\eps}},
\]
and immediately
\begin{equation}
  \label{eq:gZeps}
 Z^\eps = g^\eps(0,x_0).
\end{equation}
In other words, the value of the optimal control at the starting point of the dynamics
provides the free energy of the system. 

\begin{remark}[Stochastic solutions]
\label{rem:stochsol}
It is interesting at this point to observe that the condition for~$g^\eps$ to solve~\eqref{eq:HJB}
everywhere is too stringent to minimize the variance of~$\Act^\eps$. It is indeed
sufficient that
\begin{equation}
  \label{eq:stochsol}
  \left\{
  \begin{aligned}
    \left(\partial_t g^{\eps} + \Hc^\eps   g^{\eps} \right)(t,\Xtilde_t^\eps)  &= 0,  \\
    g^{\eps}(T,\Xtilde_T^\eps) & = f(\Xtilde_T^\eps) , &
  \end{aligned}
  \right.    
\end{equation}
for almost all~$t\in [0,T]$ and almost surely with respect to the
process~$(\Xtilde_t^\eps)_{t\in [0,T]}$ defined in~\eqref{eq:Xtilde} when the bias is~$g^\eps$ itself. 
This notion of solution based on a stochastic representation is actually presented in details
in the recent paper~\cite{leonard2021feynman}.
While~\cite[Theorem~5.9]{leonard2021feynman} is motivated by considering non-smooth
coefficients and solutions for the HJB equation through extended generators,
it motivated the current study
for its relation with variance reduction. In other words, it suggests that the optimal
control only needs to be well approximated around appropriately titled trajectories (and
not everywhere in space) for relative variance to be reduced. This is the basis of the notion
of stochastic viscosity approximation that we propose in Section~\ref{sec:SVA}. 
\end{remark}

\section{Stochastic viscosity approximations}
\label{sec:results}

We now present the main results of the paper: the definition of stochastic
viscosity approximations in Section~\ref{sec:SVA},
and the associated variance reduction property in Section~\ref{sec:vareduc}.

\subsection{Definition}
\label{sec:SVA}
The goal of this work is to find criteria for a function~$g$ to reduce the
variance of the estimator~\eqref{eq:Acteps}. For this, we saw in
Section~\ref{sec:optimal} that the optimal solution to the variance reduction problem is
the solution to a HJB equation, whose associated estimator has zero variance.
Since the solution to this equation is in general unknown, it is natural to
try to approximate it.

There exist a series of works dedicated to
approximating Hamilton--Jacobi PDEs, typically by considering decentered
finite difference schemes and assessing convergence through a Lebesgue or Sobolev distance to
the exact solution~\cite{crandall1984two,souganidis1985approximation,barles2002convergence,oberman2006convergent,barles2007error}.
Such convergence criteria rely on non-local quantities, and they are a priori
not related to the problem of variance reduction. In other word, there is no one-to-one
correspondance
between the approximation of the HJB equation and relative variance of the associated estimator.
Moreover, we are often interested in situations where the dimension~$d$ of the system is large,
so that finite difference schemes are typically not applicable.
In this context of stochastic optimal control in high dimension, a series of more
specific techniques have also been designed,
for instance~\cite{vanden-eijnden-weare:2012,hartmann2012efficient,zhang2014applications,weinan2017deep}; yet the same issue of proving that variance reduction is achieved for a given
approximation remains -- it may even deteriorate the associated
estimator~\cite{glasserman-wang:1997}.

We now propose a new notion of approximation to the HJB problem~\eqref{eq:HJB}.
Such approximate solutions will be relevant to consider the variance reduction properties
of~$\Act^\eps$ in the low temperature regime. 
This is the main contribution of this paper. 
For this, we follow the stochastic solution approach discussed in Remark~\ref{rem:stochsol}
(see in particular~\cite{leonard2021feynman} and references therein),
which suggests for~$g$ to be an approximate solution of~\eqref{eq:HJB} around
tilted trajectories.

\begin{definition}[Stochastic viscosity approximation]
  \label{def:stochapprox}
  Let Assumption~\ref{as:reg} hold,~$g^\eps$ be the solution to~\eqref{eq:HJB},
  and $x_0\in \R^d$ be the initial condition defining~$A^\eps$ in~\eqref{eq:FK}.
  Let~$\hat g^\eps$ be a function satisfying Assumption~\ref{as:novikov}
  with~$(\Xtilde_t^\eps)_{t\in[0,T]}$ the process tilted by~$\hat g^\eps$
  through~\eqref{eq:Xtilde}. We say that~$\hat g^\eps$ is a $k$-stochastic
  viscosity approximation of~\eqref{eq:HJB} if there exist~$k,\eps_0>0$,
  a family of $\R^d$-valued adapted processes~$(Q_t^\eps)_{t\in[0,T]}$ for $\eps\in (0,\eps_0]$
    (residual process), and a family of real-valued random
    variables~$Q_f^\eps$ for $\eps\in (0,\eps_0]$ (residual terminal condition)  such that the
  following conditions hold:
  \begin{itemize}
  \item \emph{Approximate HJB equation along tilted paths}:
     for any~$\eps\in (0,\eps_0]$,
    \begin{equation}
      \label{eq:approxHJB}
      \left\{
      \begin{aligned}
      (\partial_t \hat g^\eps  + \Hc^\eps \hat g^\eps ) (t,\Xtilde_t^\eps)
        & =   \eps^k Q_t^\eps,
        \\
        \hat g^\eps (T,\Xtilde_T^\eps) = f(\Xtilde_T^\eps) & + \eps^k Q_f^\eps.    
      \end{aligned}
      \right.
  \end{equation}
    for almost all $t\in [0,T]$ and almost surely with respect to~$(\Xtilde_t^\eps)_{t\in[0,T]}$.
  \item \emph{Bound}: There exists a real number~$C_Q\geq 0$ independent of~$\eps$ such
    that the family of random variables
    \begin{equation}
      \label{eq:Qeps}
    Q^\eps = Q_f^\eps + \int_0^T Q_t^\eps\,dt
    \end{equation}
    satisfies
    \begin{equation}
      \label{eq:Qepsbound}
\forall\,\eps\in(0,\eps_0],\quad    \left|\log \E\left[\e^{2Q^\eps}\right]\right|\leq 2C_Q.
    \end{equation}
\end{itemize}
\end{definition}

Let us provide some comments on Definition~\ref{def:stochapprox}.
The main condition~\eqref{eq:approxHJB} is for~$\hat g^\eps$ to be an approximate
solution of~\eqref{eq:HJB} (including the boundary conditions) along trajectories
tilted by~$\hat g^\eps$ itself. By approximate we mean that evaluating the Hamilton--Jacobi
operator along such stochastic trajectories is an appropriately bounded stochastic process
(which we may call residual process) scaled by~$\eps^k$. Since the error term vanishes in
the small temperature limit, there is a notion
of viscosity involved, as it is presented for instance in~\cite[Chapter~2, Theorem~3.1]{freidlin1998random} 
We thus understand the naming stochastic
viscosity approximation, which we abbreviate as SVA.

\begin{remark}

  \begin{itemize}
  \item We ask for~$\hat g^\eps$ to be of class~$C^{1,2}$ for computing derivatives
    in the classical sense, but using the notion of stochastic
    derivative~\cite[Section~2]{leonard2021feynman} we could consider less regular
    approximations (typically almost everywhere differentiable functions, a situation
    that may arise for piecewise defined functions). This is a nice feature of SVAs.
    \item We also emphasize that Definition~\ref{def:stochapprox} applies way beyond
    the scope of Assumption~\ref{as:reg}, typically as soon as the stochastic process~\eqref{eq:X}
    and its associated Fokker--Planck equation  make sense.
  \item We claim by no mean that the bound~\eqref{eq:Qepsbound} is optimal in any
    sense: it is simply convenient for the proofs below, and seems reasonable to check
    at least in simple situations.
  \item Finally, we assume that~\eqref{eq:approxHJB} holds almost surely. However, it
    is likely that a weaker probabilistic sense is sufficient, for instance the
    equality could be satisfied in expectation only.
  \end{itemize}
\end{remark}

\subsection{Relation to variance reduction}
\label{sec:vareduc}

As explained in Section~\ref{sec:optimal},
we have defined our notion of stochastic viscosity approximation precisely in
order to make the tilted estimator~$\Act^\eps$ close to deterministic by reducing
the random terms in the Girsanov weight to a small-noise factor. We first
show that Definition~\ref{def:stochapprox} ensures a small temperature consistency of
the estimated quantity.

\begin{lemma}
  \label{lem:consistency}
Assume that~$\hat g^\eps$ is a $k$-SVA of~$g^\eps$ for some $k \geq 1$.
Then it holds
\begin{equation}
  \label{eq:approxZ0}
\hat g^\eps (0,x_0)= g^\eps(0,x_0) + \mathrm{O}(\eps^{k}).
\end{equation}
\end{lemma}
Recalling~\eqref{eq:gZeps}, it holds $Z^\eps = g^\eps(0,x_0)$. Lemma~\ref{lem:consistency}
thus shows that a stochastic viscosity approximation provides an estimate of the quantity
of interest through its initial condition according to
\begin{equation}
  \label{eq:Zapproxg}
  Z^\eps=\hat g^\eps (0,x_0)+ \mathrm{O}(\eps^{k}).
\end{equation}
The proof is as follows.
\begin{proof}[Proof of Lemma~\ref{lem:consistency}]
  We use the Girsanov theorem~\eqref{eq:Agirsanov} to write:
  \[
\e^{\frac1\eps g^\eps(0,x_0)} =\E_{x_0} \left[ \e^{\frac1\eps f( X_T^{\eps})
    }\right] = \E_{x_0} \left[ \exp\left(\frac{1}{\eps}
    \big[f(\widetilde X_T^{\eps}) + \hat g^\eps(0, x_0) - \hat g^\eps(T,\widetilde X_T^{\eps}) \big]
    + \frac{1}{\eps} \int_0^T (\partial_t \hat g^\eps + \Hc^\eps \hat g^\eps) (t,\widetilde X_t^{\eps})\,dt
    \right)\right],
\]
where~$(\widetilde X_t)_{t\in[0,T]}$ is the process~\eqref{eq:Xtilde} for the bias
function~$\hat g^\eps$. Using~\eqref{eq:approxHJB} we obtain
\[
\e^{\frac1\eps g^\eps(0,x_0)} =
\e^{\frac1\eps \hat g^\eps(0,x_0)}
 \E_{x_0}\left[ \e^{  \eps^{k-1} \left(  Q_F^\eps
        + \int_0^T Q_t^\eps\,dt\right) }
  \right]=
\e^{\frac1\eps \hat g^\eps(0,x_0)}
 \E_{x_0}\left[ \left(\e^{  2Q^\eps}\right)^{\frac{\eps^{k-1}}{2} }
  \right].
\]
Given that $k\geq 1$, the function $x\in\R_+ \mapsto x^{\eps^{k-1}/2}$ is
concave for $\eps \leq 1$ (we may replace~$\eps_0$ by $\mathrm{min}(\eps_0,1)$ if necessary).
As a result, we can use Jensen's inequality combined with~\eqref{eq:Qepsbound} to reach
\[
|g^\eps (0,x_0) - \hat g^\eps(t,x_0)|\leq  C_Q \eps^k,
\]
which is the desired conclusion.
\end{proof}

 Using Definition~\ref{def:logef} of log-efficiency, we can now state
the variance reduction property of a $k$-stochastic viscosity
approximation.
\begin{theorem}
  \label{th:logeff}
  Assume that~$\hat g^\eps$ is a $k$-SVA of~$g^\eps$ for some $k \geq 1$.
  Then the associated estimator~$\Act^\eps$ is $k$-log efficient.
\end{theorem}

This is our main result associated to the definition of stochastic viscosity approximation.
Since  Lemma~\ref{lem:consistency} already provides the estimate~\eqref{eq:Zapproxg} of the free energy,
Theorem~\ref{th:logeff} ensures that the correction terms can be estimated
by Monte Carlo simulation with a bounded or vanishing variance.
Contrarily to Lebesgue-norm criteria that may not ensure good variance reduction properties,
we here draw a very clear connection between SVA and log-efficiency of the associated estimator.
Let us present the proof, which is rather straightforward.

\begin{proof}[Proof of Theorem~\ref{th:logeff}]
  For this proof, we write
  \[
  \hat \Gc^\eps =  \exp\left(\frac{1}{\eps}
    \big[ \hat g^\eps(0, x_0) - \hat g^\eps (T,\widetilde X_T^{\eps}) \big]
    + \frac{1}{\eps} \int_0^T (\partial_t \hat g^\eps + \Hc^\eps \hat g^\eps)(t,\widetilde X_t^{\eps})\,dt
    \right)
    \]
    for the Girsanov weight associated with the function~$\hat g^\eps$.
  Let us consider first the quantity
  \[
  \E_{x_0}\left[ \e^{\frac 2 \eps f(\Xtilde_T^\eps) }(  \hat \Gc^\eps)^2
    \right]=\e^{\frac2\eps \hat g^\eps(0,x_0)}
    \E_{x_0}\left[ \e^{\frac 2 \eps \left( f(\Xtilde_T^\eps) -\hat g^\eps(T,\Xtilde_T^\eps)
        + \int_0^T(\partial_t \hat g^\eps + \Hc^\eps \hat g^\eps)(t,\Xtilde_t^\eps)\,dt\right) }
  \right].
    \]
    The approximate HJB equation~\eqref{eq:approxHJB}
    (including the terminal condition) in Definition~\ref{def:stochapprox}
    implies again that
    \[
    \E_{x_0}\left[ \e^{\frac 2 \eps ( f(\Xtilde_T^\eps) - \hat g^\eps(T,\Xtilde_T^\eps)
        + \int_0^T(\partial_t \hat g^\eps + \Hc^\eps \hat g^\eps)(t,\Xtilde_t^\eps)\,dt }
      \right] = \E_{x_0}\left[ \e^{\frac 2 \eps \left( \eps^k Q_F^\eps
        + \eps^k \int_0^T Q_t^\eps\,dt\right) }
  \right].
    \]
    Using the  consistency relation~\eqref{eq:approxZ0}
    proved in Lemma~\ref{lem:consistency} and
    recalling the expression~\eqref{eq:Qeps} of the residual~$Q^\eps$, we obtain
    \[
  \E_{x_0}\left[ \e^{\frac 2 \eps f(\Xtilde_T^\eps) }(  \hat \Gc^\eps)^2
    \right]=\e^{\frac2\eps \left(g^\eps(0,x_0) + \mathrm{O}(\eps^k)\right)}
  \E_{x_0}\left[\e^{2 \eps^{k-1} Q^\eps}\right].
  \]
  Since $k\geq 1$, we use again the concavity of the function $x\in\R_+ \mapsto x^{\eps^{k-1}}$
  for $\eps \leq 1$.
  Together with Jensen's inequality and~\eqref{eq:Qepsbound} we obtain,
  for $\eps\in(0, \eps_0]$:
  \[
  \E_{x_0}\left[\e^{2\eps^{k-1} Q^\eps}\right]
  \leq \E_{x_0}\left[\e^{2 Q^\eps}\right]^{\eps^{k-1}}\leq \e^{2\eps^{k-1} C_Q}.
  \]
  
  Finally, by~\eqref{eq:gZeps}, it holds
  $\E\left[\e^{\frac 1 \eps f(X_T^\eps)} \right]^2 = \e^{\frac 2 \eps g^\eps(0,x_0)}$,
  so that we obtain the bound
  \[
  \rho(\eps) \leq
    \frac{\e^{\frac2\eps (g^\eps(0,x_0) + \mathrm{O}(\eps^k))} \e^{2\eps^{k-1} C_Q}
    }{\e^{\frac 2 \eps g^\eps(0,x_0)}} 
  = \e^{2C_Q \eps^{k-1} +\mathrm{O}(\eps^{k-1})}.
  \]
 This implies that
  \[
R(\varepsilon) = \eps \log \rho(\eps)  = \mathrm{O}(\eps^k)
\]
as $\eps\to 0$, which is the desired result.
\end{proof}

Note that we cover the case $k\geq 1$ for simplicity and because this is
the main situation of interest (since relative variance is a priori not even bounded
for $k<1$), but one can prove a  similar result for $k\in[0,1)$ under stronger assumptions,
  for instance by assuming that~$Q^\eps$ is bounded. As discussed after Definition~\ref{def:logef},
  in the particular case where~$k\geq 1$,
the relative variance $\rho(\eps) - 1$ is bounded, hence we ensure that the variance
of numerical simulations remains finite even in the vanishing temperature regime,
which is what we are looking for in practice.

\section{Instanton and stochastic viscosity approximations}
\label{sec:instantonexpand}
In this section, we propose to construct a $1$-SVA and a $2$-SVA
from the theory of low temperature reaction paths called
instantons~\cite{grafke-vanden-eijnden:2019,ferre2020approximate} and we provide
illustrative numerical results.
We first present and motivate the instanton dynamics in Section~\ref{sec:lowtemp}.
We then show in Section~\ref{sec:order1} that the earlier
proposed sampling technique~\cite{grafke-vanden-eijnden:2019} relying on these equations
provides a $1$-SVA of the optimal control under appropriate assumptions.
It is thus associated with a $1$-log efficient estimator. This provides a mathematical
variance reduction property for this importance sampling scheme under geometric
conditions, which had not been proved so far.
In Section~\ref{sec:order2} we consider the next order expansion
proposed in~\cite{ferre2020approximate}, which
is a $2$-SVA of the optimal control, hence improving on variance reduction.
We finally propose in Section~\ref{sec:numerical} a simple numerical application
to illustrate the validity of our predictions. 

\subsection{Low temperature asymptotics}
\label{sec:lowtemp}
We start by presenting low temperature asymptotics concerning
the dynamics that we will use to practically build SVAs.
Recalling that $\Hc^0\varphi = b\cdot \nabla \varphi + |\sigma^T \nabla \varphi|^2 /2$,
the zero temperature HJB equation reads
\begin{equation}
  \label{eq:HJB0}
  \left\{
  \begin{aligned}
    (\partial_t g^0  + \Hc^0 g^0)(t,x)  &= 0, & \forall\,(t,x)\in [0,T)\times \X, \\
      g^0(T,x) & = f(x) , & \forall\,x\in\X.
  \end{aligned}
  \right.
\end{equation}
This is a first order nonlinear PDE for which a smooth solution often
does not exist, because of the lack of diffusive regularization. However,
a viscosity solution can most often be described through the
relation~\cite[Chapter~2, Theorem~3.1]{freidlin1998random}:
\[
\forall\, (t,x)\in[0,T]\times \X,\quad
g^0(t,x) = \lim_{\eps \to 0}\ g^\eps(t,x).
\]

In some cases, mostly when the function~$g^0$ is smooth, we can describe
the solution through its characteristics system. The same methodology
applied in a large deviations context leads to the instanton
dynamics, which is described by a forward-backward
system~\cite{grafke-vanden-eijnden:2019,ferre2020approximate} with
$\R^d$-valued position~$(\phi_t)_{t\in[0,T]}$ and momentum~$(\theta_t)_{t\in[0,T]}$
solution to
\begin{equation}
  \label{eq:instanton}
  \left\{
    \begin{aligned}{}
      &\dot \phi_t = b(\phi_t) + D \theta_t, \qquad& &\phi_0 = x_0,\\
      &\dot \theta_t = -(\nabla b)(\phi_t)\cdot \theta_t ,& &\theta_T = \nabla f(\phi_T).
    \end{aligned}
  \right.
\end{equation}
The above formula is a definition of the instanton, which can be derived
by minimizing the Freidlin--Wentzell rate function under
constraints (see~\cite{grafke-vanden-eijnden:2019} and references therein).
This dynamics describes the path realizing fluctuations in the low temperature limit.

There are various ways to introduce the instanton dynamics~\eqref{eq:instanton}.
Let us provide a maybe unusual, PDE-oriented motivation for this object, by introducing
the function
\begin{equation}
  \label{eq:g0}
  \forall\,(t,x)\in[0,T]\times \X,\quad
\hat g_1^\eps(t,x) =f(\phi_T) - \frac12 \int_t^T \theta_s \cdot D \theta_s \, ds
+ \theta_t\cdot (x - \phi_t).
\end{equation}
Note that we keep the~$\eps$ dependency on~$\hat g_1^\eps$ although it does not
depend on~$\eps$, because such an approximation may depend on~$\eps$ in general.
This function is made of three parts: a constant term, a time-dependent only
term, and a linear term in the variable~$(x - \phi_t)$. We may thus call~$\hat g_1^\eps$
a time-inhomogeneous first order polynomial in the variable~$(x - \phi_t)$.

To understand better this function, we first see that
$\hat g_1^\eps(T,x) = f(\phi_T) + \nabla f (\phi_T)\cdot (x - \phi_T)$,
which is a linearization of~$f$ around the  instanton's terminal point~$\phi_T$.
Let us now compute the value of the HJB  operator $\partial_t + \Hc^\eps$ over~$\hat g_1^\eps$.
Noting that $\nabla \hat g_1^\eps (t,x)=\theta_t$
and $\theta_t \cdot D\theta_t = |\sigma^T \theta_t|^2$ 
we get, for any $(t,x)\in[0,T]\times \X$: 
\[
\begin{aligned}
\left(\partial_t\hat g_1^\eps  + \Hc^\eps\hat g_1^\eps\right)(t,x)
 & =
 \frac12 |\sigma^T \theta_t|^2 + \dot \theta_t \cdot (x - \phi_t)
- \theta_t \cdot\dot \phi_t + b(x)\cdot \theta_t + \frac12 |\sigma^T \theta_t|^2
\\ & =  |\sigma^T \theta_t|^2- \nabla b (\phi_t) \cdot (x - \phi_t)
- \theta_t \cdot (b(\phi) + D\theta_t) + b(x)\cdot \theta_t
\\ & = \theta_t\cdot \big( b(x) - b(\phi_t) -  \nabla b (\phi_t) \cdot (x - \phi_t)\big)
= \mathrm{O}\big( (x-\phi_t)^2 \big),
\end{aligned}
\]
where we used~\eqref{eq:instanton} for the second line, while the
last inequality is derived by Taylor expansion.
We actually obtain from the above computation that
\begin{equation}
  \label{eq:HJB0hat}
  \left\{
  \begin{aligned}
    (\partial_t \hat g_1^\eps  + \Hc^\eps\hat g_1^\eps)(t,\phi_t)  &= 0, & \forall\, t\in [0,T), \\
      \hat g_1^\eps(T,\phi_T) & = f(\phi_T).&
  \end{aligned}
  \right.
\end{equation}
In other words,~$\hat g_1^\eps$ solves the HJB problem~\eqref{eq:HJB}
along the instanton.

Since our discussion in Remark~\ref{rem:stochsol} suggests that
importance sampling only requires approximation of HJB around some trajectories
and~$(\phi_t)_{t\in[0,T]}$ describes the most likely fluctuation path as $\eps\to 0$,
we understand that the ansatz~\eqref{eq:g0} can serve as a first guess for
approximating the solution~$g^\eps$ of the finite temperature equation~\eqref{eq:HJB}
along most likely fluctuation paths,
and hence reduce variance.

\subsection{Instanton bias}
\label{sec:order1}

We consider the approximation~\eqref{eq:g0} under the following assumption
(see Remark~\ref{rem:asdiscuss} for further discussion).
\begin{assumption}
  \label{as:instanton}
  The dynamics~\eqref{eq:instanton} admits a unique solution $(\phi_t,\theta_t)_{t\in[0,T]}$,
  which is of class~$C^1([0,T])^d\times C^1([0,T])^d$.
\end{assumption}
Under Assumption~\ref{as:instanton}, the function~$\hat g_1^\eps$ is well-defined
and we denote by~$(\Xtilde_t^{\eps,1})_{t\in[0,T]}$ the dynamics biased by~$\hat g_1^\eps$
as in~\eqref{eq:Xtilde}. 
Compared to the control proposed in the previous work~\cite{ferre2020approximate},
we add a constant and a time-only
dependent components (the first two terms in~\eqref{eq:g0}). This is not crucial
from a practical perspective since
the biasing force~$\nabla \hat g_1^\eps$ does not depend on these terms,
but it makes the entire approximation argument more elegant.

Before proving that~$\hat g_1^\eps$ is a $1$-stochastic viscosity approximation
of~$g^\eps$,
we first show that a trajectory tilted by~$\hat g_1^\eps$ is a perturbation of
the instanton -- very much in the spirit of the expansions that can be found
in~\cite[Chapter~2]{freidlin1998random}.

\begin{lemma}
\label{lem:Xtapproxphit}
Under Assumptions~\ref{as:reg} and~\ref{as:instanton},
there exist an adapted process~$(Y_t^\eps)_{t\in[0,T]}$ and a constant~$ C_T>0$ such that
  \begin{equation}
    \label{eq:consistency}
    \forall\,t\in[0,T],\quad
    \Xtilde_t^{\eps,1} = \phi_t + \sqrt \eps Y_t^\eps,
  \end{equation}
  and for all $t\in[0,T]$ the following bound holds 
  \begin{equation}
    \label{eq:Ytbound}
    \forall \, \eps>0,\quad |Y_t^\eps|\leq C_T \sup_{0\leq s \leq T} \,|\sigma^T B_s|,
  \end{equation}
   almost surely with respect to realizations of the Brownian motion.
\end{lemma}

The proof follows a standard Gronwall argument.

\begin{proof}[Proof of Lemma~\ref{lem:Xtapproxphit}]
 Let~$\eps>0$ be arbitrary.
Since the solution to~\eqref{eq:instanton}
is unique, we can define for all $t\in[0,T]$ the adapted process\footnote{We mention that
  in the case of multiple instantons solution to~\eqref{eq:instanton}, one could consider defining
several approximations centered on each instanton with their residual processes.}
\begin{equation}
  \label{eq:Ytproof}
  Y_t^\eps =\frac{\Xtilde_t^{\eps,1} - \phi_t}{\sqrt{\eps}}.
\end{equation}
Since $\nabla g(t,x) = \theta_t$ for all $(t,x)\in[0,T]\times \X$,
 Itô formula shows that
\[
\begin{aligned}
\sqrt \eps dY_t^\eps & = b(\Xtilde_t^{\eps,1}) + D\theta_t\, dt + \sqrt \eps \sigma^T \, dB_t
-\dot \phi_t \, dt
\\ & = \big( b (\Xtilde_t^{\eps,1}) - b(\phi_t)\big) +  \sqrt \eps \sigma^T \, dB_t,
\end{aligned}
\]
where we used~\eqref{eq:instanton}.
Integrating now in time and taking the absolute value leads to
\[
\forall\,t\in[0,T],\quad
\sqrt \eps |Y_t^\eps - Y_0^\eps|\leq
\int_0^t |b (\Xtilde_s^{\eps,1}) - b(\phi_s)|\,ds + \sqrt \eps |\sigma^T B_t|.
\]
Recalling that $Y_0^\eps = 0$ and that~$b$ is globally Lipschitz 
with constant~$\Clip$ (by Assumption~\ref{as:reg}), we obtain
\[
\forall\,t\in[0,T],\quad
\sqrt \eps |Y_t^\eps|\leq \int_0^t\Clip |\Xtilde_s^{\eps,1} - \phi_s|\,ds + \sqrt \eps  |\sigma^T B_t|,
\]
which rewrites
\[
\forall\,t\in[0,T],\quad
|Y_t^\eps|\leq \Clip \int_0^t |Y_s^\eps|\,ds +   |\sigma^T B_t|.
\]
Gronwall lemma then implies that
\[
\forall\,t\in[0,T],\quad
|Y_t^\eps|\leq \sup_{0\leq s\leq T} |\sigma^T B_s|\, \e^{T \Clip},
\]
which holds for any continuous realization of the Brownian motion, hence
almost surely, and for any $\eps>0$. We thus have the bound~\eqref{eq:Ytbound},
while~\eqref{eq:consistency} follows from the definition~\eqref{eq:Ytproof}
of~$Y_t^\eps$. 
\end{proof}

Proving that~$\hat g_1^\eps$ is a stochastic viscosity approximation of~$g^\eps$
is now a matter of Taylor expansion under some technical conditions.

\begin{theorem}
  \label{th:1stoch}
  Let~$\lambda$ be the largest eigenvalue of $D$.
  Let moreover Assumptions~\ref{as:reg} and~\ref{as:instanton} hold and suppose that
  there exist $\delta_f,\delta_b\geq 0$ such that
  \begin{equation}
    \label{eq:condelta}
    \delta_f + \delta_b T < \frac{1}{\lambda TC_T^2},    
  \end{equation}
  and:
  \begin{itemize}
  \item The function~$f$ satisfies
  \begin{equation}
    \label{eq:Hessianf}
  \forall\, y \in\R^d,\quad
  \sup_{x\in\X}\ y\cdot \nabla^2 f(x) y \leq \delta_f|y|^2.
  \end{equation}    
\item The field~$b$ satisfies
  \begin{equation}
  \label{eq:Hessianb}
  \forall\, y \in\R^d,\quad\forall\, t\in[0,T],\quad
  \sup_{x\in\X}\ y\cdot \big(\theta_t\cdot \nabla^2 b(x)\big) y \leq \delta_b|y|^2.
\end{equation}
  \end{itemize}
  Then the function~$\hat g_1^\eps$ defined in~\eqref{eq:g0} satisfies
  Assumption~\ref{as:novikov} and is a $1$-SVA of the HJB problem.
  As a result, it defines a $1$-log efficient estimator.
\end{theorem}
The proof of Theorem~\ref{th:1stoch} is presented just below.
We may interpret the geometric conditions~\eqref{eq:Hessianf}-\eqref{eq:Hessianb} as
complementary to Assumption~\ref{as:instanton}. First,~\eqref{eq:Hessianf} asks for~$f$
to be <<mostly concave>>. Typically, we may expect problems to occur when~$f$ has several
pronounced maxima, so that the maximum value of~$f$ can be reached in several ways.
A similar condition~\eqref{eq:Hessianb} holds for the drift~$b$: if we further assume
that~$\nabla f$ has nonnegative coordinates, then~$\theta_t$ has nonnegative coordinates
at least for small times, and so~\eqref{eq:Hessianb} is a concavity condition on~$b$.
Obviously, these conditions are quite drastic and could be relaxed if one proves
additional stability properties of the process or has other information on the drift~$b$
(for instance bounded second derivative). 

\begin{proof}[Proof of Theorem~\ref{th:1stoch}]

We check the approximate HJB equation~\eqref{eq:approxHJB} with bound~\eqref{eq:Qepsbound}
by relying on Lemma~\ref{lem:Xtapproxphit}.
Consider the terminal condition first.  Using the terminal value of~$\theta_T$
in~\eqref{eq:instanton}, we have
  \[
  \hat g_1^\eps(T,\Xtilde_T^{\eps,1}) = f (\phi_T) +\nabla f (\phi_T)\cdot( \Xtilde_T^{\eps,1} - \phi_T).
  \]
  Taylor's integral theorem applied at first order to~$f$ around~$\phi_T$ shows,
  using  Lemma~\ref{lem:Xtapproxphit}, that we have almost surely
  \[
  \begin{aligned}
  Q_f^\eps  =\eps^{-1}\big(
f(\Xtilde_T^{\eps,1}) -   \hat g_1^\eps(T,\Xtilde_T^{\eps,1})\big)
 & =\eps^{-1}\big(  f(\Xtilde_T^{\eps,1}) -f (\phi_T)- \nabla f (\phi_T)
\cdot( \Xtilde_T^{\eps,1} - \phi_T)\big)\\ &  = \frac 12 \int_0^1
 (1-s) Y_T^\eps\cdot\nabla^2 f\big( \phi_T +s \sqrt \eps Y_T^\eps \big) Y_T^\eps\,ds.
\end{aligned}
\]
According to~\eqref{eq:Hessianf}, it holds
\[
Q_f^\eps\leq \delta_f \frac{|Y_T^\eps|^2}{4},
\]
which we will use to control the residual process\footnote{Although we impose
  a constraint on~$\nabla^2 f$ over the whole space, we see that it is essentially
  important to control this quantity around~$\phi_T$. One difficulty however
  comes from the fact that~$\phi_T$
  itself depends on the function~$f$. This makes a general condition for~$f$ to satisfy the
  bound hard to state at this level of generality.
  The same is true for the condition on~$\nabla^2 b$.}.

Let us now check the dynamical part in~\eqref{eq:approxHJB}. First, we easily compute
  \[
  \begin{aligned}
    \partial_t \hat g_1^\eps (t,x) & =  \frac12 \theta_t \cdot D \theta_t
    + \dot \theta_t\cdot (x - \phi_t) - \theta_t \cdot \dot \phi_t,
    \\
    \nabla \hat g_1^\eps(t,x) & = \theta_t,
    \\
    \nabla^2  \hat g_1^\eps(t,x) & = 0.
  \end{aligned}
  \]
  Therefore, using Lemma~\ref{lem:Xtapproxphit} and equation~\eqref{eq:instanton} we obtain
  \[
  \begin{aligned}
    (\partial_t  \hat g_1^\eps + \Hc^\eps  \hat g_1^\eps)
    (t,\Xtilde_t^{\eps,1}) & =  \frac12 \theta_t \cdot D \theta_t
  - \sqrt\eps\theta_t \cdot \nabla b(\phi_t) Y_t^\eps - (b(\phi_t) + D\theta_t) \cdot \theta_t
  + b(\Xtilde_t^{\eps,1})\cdot \theta_t +\frac12 \theta_t \cdot D \theta_t.
  \\ & = - \sqrt\eps\theta_t \cdot \nabla b(\phi_t) Y_t^\eps  
  + (b(\Xtilde_t^{\eps,1})- b(\phi_t))\cdot \theta_t
  \\ & = \theta_t\cdot \left(
  b(\Xtilde_t^{\eps,1})- b(\phi_t) - \nabla b(\phi_t) (\Xtilde_t^{\eps,1} - \phi_t)  
  \right).
  \end{aligned}
  \]
  Finally, by defining the process
  \[
  Q_t^\eps =\frac{\theta_t}{\eps} \cdot\big(  b(\Xtilde_t^{\eps,1})- b(\phi_t)
  - \nabla b(\phi_t)(\Xtilde_t^{\eps,1} - \phi_t)  
  \big),
\]
we have 
\[
  (\partial_t \hat g_1^\eps + \Hc^\eps \hat g_1^\eps)(t,\Xtilde_t^{\eps,1}) = \eps Q_t^\eps.
\]
In order to obtain~\eqref{eq:Qepsbound} and conclude, it
remains to bound~$Q_t^\eps$. We apply again Taylor's
theorem at first order between~$\Xtilde_t^{\eps,1}$ and~$\phi_t$ for all time~$t\geq0$
and any continuous realization of~$(\widetilde X_t^\eps)_{t\in[0,T]}$, now to the
$\R^{d}$-valued function~$ b$, to obtain 
\[
\forall\,t\in[0,T],\quad Q_t^\eps =Y_t^\eps\cdot\left( \frac 12 \int_0^1
 (1-s)\theta_t\cdot \nabla^2 b\big( \phi_t +s \sqrt \eps Y_t^\eps \big)\,ds\right) Y_t^\eps.
\]
Finally, the condition~\eqref{eq:Hessianb} ensures that
\[
Q_t^\eps \leq \frac{\delta_b}{4}|Y_t^\eps|^2,
\]
and the total residual~$Q^\eps$ defined in~\eqref{eq:Qeps} satisfies:
\[
Q^\eps \leq  \frac{\delta_f}{4}|Y_T^\eps|^2 + \frac{\delta_b}{4}\int_0^T |Y_t^\eps|^2\, dt, 
\]
Using~\eqref{eq:Ytbound} in Lemma~\ref{lem:Xtapproxphit} then leads to
\[
Q^\eps \leq C_T^2\sup_{0\leq s \leq T}|\sigma^T B_s|^2 \left( \frac{\delta_f}{4} + \frac{\delta_b}{4}T\right).
\]
In order to find conditions for the integration condition~\eqref{eq:Qepsbound} to hold,
we recall that~$\lambda>0$ is the largest eigenvalue of $D=\sigma \sigma^T$, and introduce
\[
c = 2 \lambda C_T^2 \left( \frac{\delta_f}{4} + \frac{\delta_b}{4}T\right).
\]
Since $|\sigma^T B_s|^2 = B_s^T D B_s$, we have
\[
2Q^\eps \leq c\sup_{0\leq s \leq T}| B_s|^2.
\]
It is simple to show that
\[
\sup_{0\leq s \leq T}|B_s|^2 = \left(\sup_{0\leq s \leq T}| B_s|\right)^2.
\]
On the other hand, the law of $ \sup_{0\leq s \leq T}| B_s|$ is described
by~\cite[Formula~1.1.8 on page~250]{borodin2015}:
\begin{equation}
  \label{eq:cdf}
\forall\, y,x\geq 0, \quad \proba_x\left(\sup_{0\leq s \leq T}| B_s|\geq y\right)
= \sum_{k=-\infty}^{+\infty} (-1)^k\mbox{sign}(x +2(k+1)y)
\mbox{Erfc} \left( \frac{|x+(2k+1)y|}{\sqrt2T}\right),
\end{equation}
where
\[
\forall\,x\in\R,\quad \mbox{Erfc}(x) = \frac{2}{\sqrt{\pi}}\int_x^{+\infty}
\e^{-u^2}du.
\]
Using the case $x=0$ in~\eqref{eq:cdf}, we conclude that
\[
\E\left[ \e^{2Q^\eps}\right]\leq \E\left[ \e^{c\left(\sup_{0\leq s \leq T}| B_s|\right)^2}\right]
\]
is finite independently of~$\eps>0$ when
$c < 1/(2T)$, which is equivalent to~\eqref{eq:condelta}.
As a result,~\eqref{eq:Qepsbound} holds provided~$\delta_f$ and~$\delta_b$ are small enough
for~\eqref{eq:condelta} to hold.

Since~$\nabla \hat g_1^\eps$ does not depend on the position variable,
Novikov's condition~\eqref{eq:novikov} holds and so does Assumption~\ref{as:novikov}.
As a result, all the conditions of
Definition~\ref{def:stochapprox} are satisfied and~$\hat g_1^\eps$ is a $1$-SVA
of the HJB problem~\eqref{eq:HJB}. According to~Theorem~\ref{th:logeff}, the
associated estimator is $1$-log efficient, which concludes the proof.

\end{proof}

\begin{remark}[About Assumption~\ref{as:instanton}]
  \label{rem:asdiscuss}
  While Assumptions~\ref{as:reg} and~\ref{as:novikov} are of technical nature and
  could easily be relaxed, Assumption~\ref{as:instanton} is of geometrical nature, and has a
    crucial importance.
    It is easy to build examples for which all quantities are smooth,
    but where this assumption does not hold~\cite{glasserman-wang:1997}, typically for
    symmetry or convexity reasons. An example is to set $d=1$, $b=0$, $\sigma = 1$,
  $f(x) = \tanh(|x|)$ and $x_0 = 0$: in this case the instanton may maximize the
    final value of~$f$ either by going to the right or to the left for the same
    kinetic cost, so there is no uniqueness of the instanton. This is a situation
    similar to the one appearing when solving first order PDEs with the method
    of characteristics.

    In this case, it is natural to divide the space~$\R^d$ into different regions
    where one instanton would dominate. An idea in this direction is to resort
    to the so-called \emph{regions of strong regularity} defined in~\cite{fleming1992asymptotic}.
    The idea is to split the space in~$N$ regions~$(\Gamma_i)_{i=1}^N$ such that~$[0,T]\times \R^d = \cup_{i=1}^N \Gamma_i$
    with $\forall\,i\neq j$, $|\Gamma_i\cap\Gamma_j|=0$, and  such that the problem can be treated independently in each
    region. We would then define~$N$ instantons and their approximations~$(\hat g^{\eps,i})_{i=1}^N$,
    and a natural candidate approximation is
    \[
    \hat g^{\eps}(t,x) =
    \hat g^{\eps,i}(t,x)\quad \mbox{if}\quad (t,x)\in\Gamma_i \quad \mbox{and}\quad x\notin\Gamma_j,\ \forall\,j\neq i.
    \]
    Since we assume the boundaries of these sets to have zero Lebesgue measure, this function is well-defined
    almost everywhere. Of course, using a finite polynomial in each region,~$\hat g^{\eps}$
    could have singularities on the boudaries of the sets,
    which might be a side effect of going through a low temperature expansion. 
  \end{remark}

\subsection{Riccati bias}
\label{sec:order2}

One motivation for the current work was to better understand the approximation
proposed in~\cite{ferre2020approximate}. The point of the authors was to
consider~$\hat g_1^\eps$ as the first term of a time-dependent Taylor expansion
in the argument~$(x - \phi_t)$, which naturally suggests to continue the series with higher orders.
Based on this earlier work~\cite[eq.~(21)]{ferre2020approximate}, we introduce
the following quadratic approximation, for a $R^{d\times d}$-valued
dynamics~$(K_t)_{t\in[0,T]}$,
\begin{equation}
  \label{eq:g2}
  \hat g_2^\eps(t,x) = f(\phi_T) - \frac12 \int_t^T \theta_s \cdot D \theta_s\,ds
  +\frac\eps2 \int_t^T  D : K_s\,ds
  + \theta_t\cdot ( x - \phi_t) + \frac12  ( x - \phi_t) \cdot K_t  ( x - \phi_t).
\end{equation}
Like for the expansion performed in the proof of
Theorem~\ref{th:1stoch}, we can consider the above ansatz and identify the
 equations associated with each term to obtain a $2$-SVA. This is exactly the computation performed
in~\cite[Section~3 and Appendix~B]{ferre2020approximate}, which we don't reproduce here
for conciseness. The obtained result is that we shall take~$(K_t)_{t\in[0,T]}$ solution
to the following Riccati system:
\begin{equation}
  \label{eq:EDOK}
  \left\{\begin{aligned}
 \dot{K}_t + (\nabla b)^{\mathrm{T}} K_t + K_t^{\mathrm{T}}
  \nabla b +  \nabla^2 b \cdot \theta_t + K_t^{\mathrm{T}} D K_t & = 0,
  \\ K_T   & = \nabla^2 f(\phi_T).
  \end{aligned}
  \right.
\end{equation}
By imposing~$(\phi_t,\theta_t)_{t\in [0,T]}$ to be solution to~\eqref{eq:instanton}
and~$(K_t)_{t\in[0,T]}$ to be solution to~\eqref{eq:EDOK} (assuming well-posedness
of these systems, which is true at least for small times) while assuming appropriate 
bounds on~$\nabla^3 f$ and~$\nabla^3 b$, we can thus show that~$\hat g_2^\eps$ is a $2$-SVA of~$g^\eps$.
As mentionned above, the computation performed
in~\cite[Section~3 and Appendix~B]{ferre2020approximate} follows the same Taylor expansion
as for Theorem~\ref{th:1stoch} but one order beyond, which naturally leads to
a Riccati equation. Since~$\hat g_2^\eps$ defines a $2$-SVA, the associated
estimator is 2-log efficient. This entails that the second order approximation proposed
in~\cite{ferre2020approximate} indeed improves on variance reduction compared to the
first order one earlier proposed in~\cite{grafke-vanden-eijnden:2019}. For a given~$\eps$,
we expect to find an error one order of magnitude smaller for a $2$-SVA compared to a $1$-SVA,
which is a major improvement.

\subsection{Numerical application}
\label{sec:numerical}

We now illustrate our theoretical findings by proposing a simple situation in
which all our conditions are satisfied and our results apply. We then perform a
numerical simulation that confirms our predictions.

sThe system is a simple one dimensional Ornstein--Uhlenbeck process, \textit{i.e.}
$d=1$, $\sigma = 1$ and $b(x) = - x$. In order to propose a smoothed version of
a probability like~$\mathbb{P}( X_T^\eps > 1)$, we set $f(x) = - (x-2)^4/4$. This
weight function gives larger importance to those trajectories ending around~$2$. We also
set $x_0 = -1$, so tilted trajectories typically relax towards~$0$ before being
pushed away to larger positive values. This
system is non-trivial while satisfying Assumption~\ref{as:reg}. In this case the
solution to the instanton equation~\eqref{eq:instanton} is actually explicit,
while the Riccati equation~\eqref{eq:EDOK} is stable and can be integrated with
a simple Euler scheme.
For more general dynamics, the instanton can be computed with the algorithm
described in~\cite[Section~III~A]{grafke-vanden-eijnden:2019} (which 
we actually use in practice). Finally, $f''(x)\leq 0$ and $b''(x)=0$ for all $x\in\R$,
so the conditions~\eqref{eq:Hessianf} and~\eqref{eq:Hessianb} are satisfied and
Theorem~\ref{th:1stoch} applies. As a result,~$\hat g_1^\eps$ defines a $1$-SVA
of the system. Similarly, since the system is quite simple, we can show that~$\hat g_2^\eps$
is a $2$-SVA.

In order to perform numerical simulations, we set $T=5$ and discretize the SDE,
the instanton and Riccati equations with a time step $\Delta t = 5\times 10^{-3}$. We numerically
assess that this value is small enough to neglect the bias due to time discretization
compared to the relative variance we are interested in (not shown).
We draw $5\times10^6$ trajectories to estimate the relative variance ratio~$R(\eps)$
defined in~\eqref{eq:Rcoeff} for a series of values of~$\eps$. The left panel
in Figure~1 shows the evolution of the instanton and Riccati terms with time.
While momentum adds an upward drift, the negative Riccati term forces the process to remain around
the instanton. On the right panel we show the evolution of~$R(\eps)$
with~$\eps$ for the different estimators. We observe that~$R(\eps)$
is roughly constant without biasing, meaning relative variance grows exponentially
in this case. On the other hand,
the linear and quadratic decay at log scale for the first and second order biases
 are respectively in accordance with the results of
Theorem~\ref{th:1stoch} and the discussion in Section~\ref{sec:order2}.

\begin{figure}[h]
  \label{fig:OU}
  \includegraphics[width=0.48\textwidth]{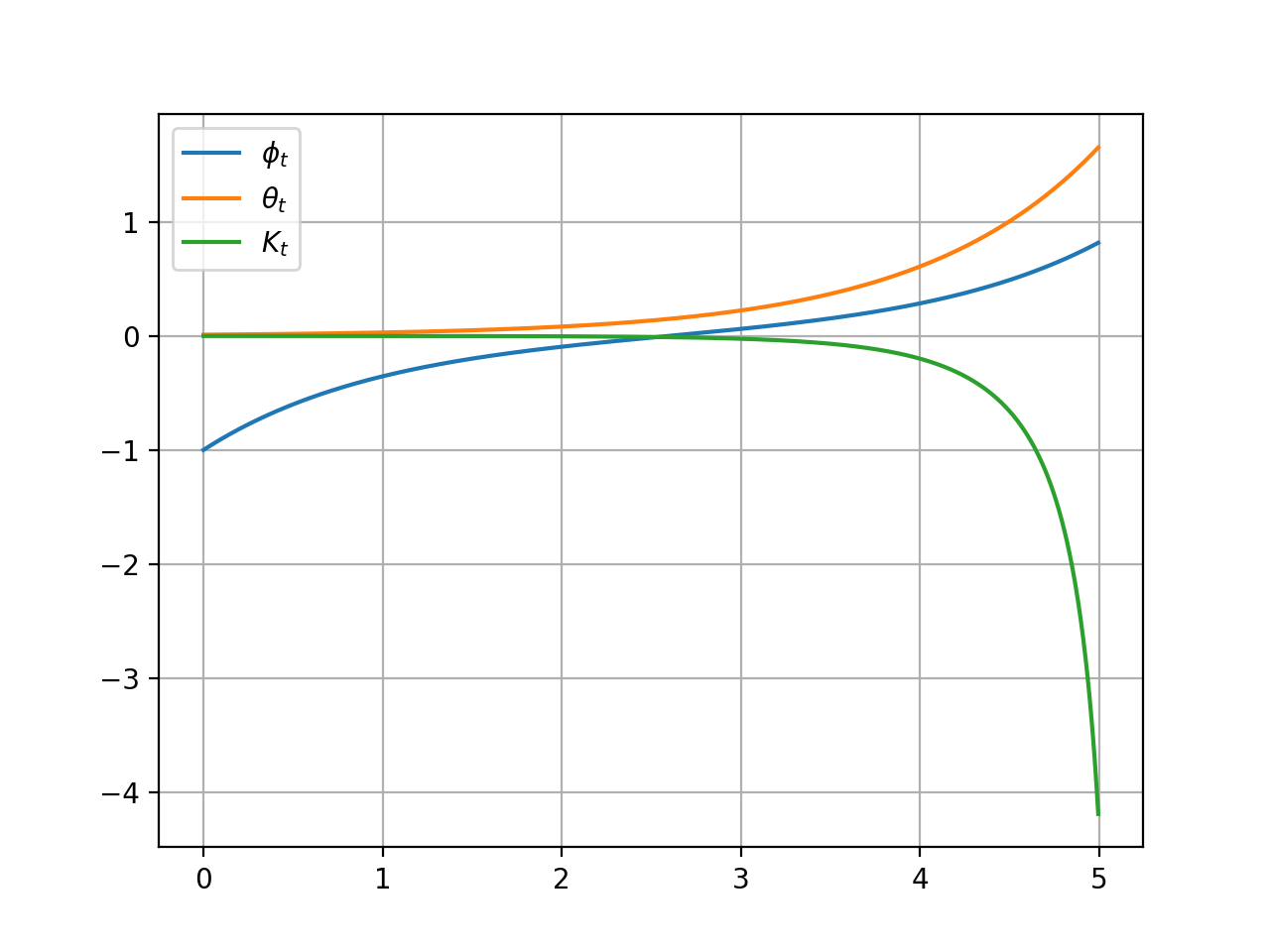}
  \includegraphics[width=0.48\textwidth]{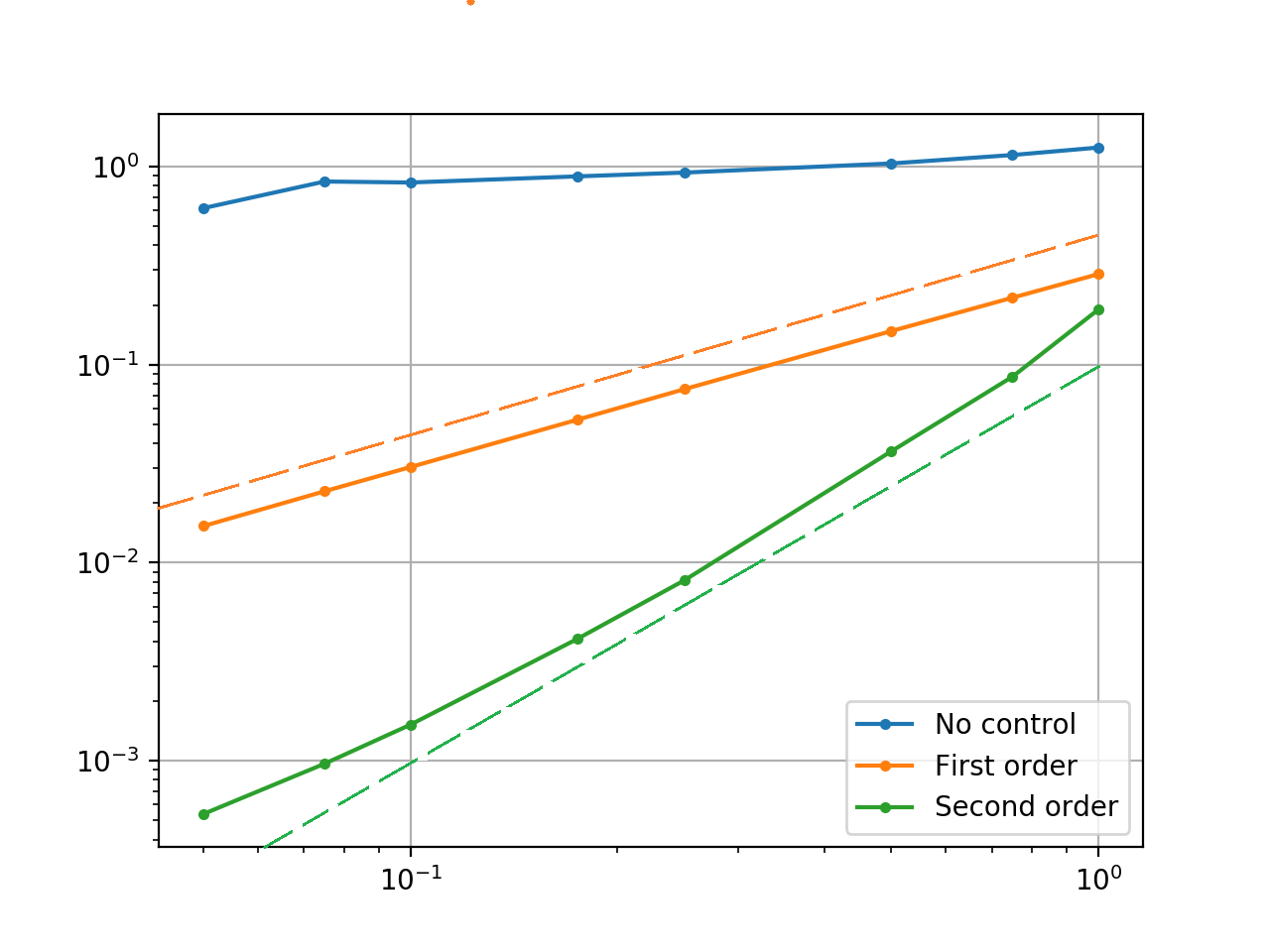}
  \caption{Left: Instanton, momentum and Riccati matrix (here scalar values).
    Right: Evolution of~$R(\eps)$ as defined in~\eqref{eq:Rcoeff} with~$\eps$ in logarithmic scale,
    for the unbiased estimator and the dynamics biased by~$\hat g_1^\eps$ and~$\hat g_2^\eps$. Linear
  and quadratic decays are plotted in dashed lines.}
\end{figure}

\section{Discussion}
\label{sec:discussion}

This work is concerned with the issue of variance reduction for the Monte Carlo
estimation of exponential-like expectations, as they often arise in statistical
physics and other areas. Our main goal was to design a framework
to assess whether a given approximation of the optimal Hamilton--Jacobi control
indeed reduces relative variance as temperature becomes small, and at which scale.

To achieve this, we introduce the notion of stochastic viscosity approximation
(SVA), inspired by recent developments on the theory of Feynman--Kac partial differential
equations~\cite{leonard2021feynman}. We believe our definition is meaningful:
a SVA approximately solves the Hamilton--Jacobi--Bellman equation along relevant tilted
paths. This is precisely what is
needed in the Girsanov theorem to make stochastic terms small. We therefore prove
an associated variance
reduction property. This definition comes in sharp
contrast with standard approximation techniques criteria, which typically control a non-local
error to assess convergence as a mesh size decays to zero. Here we show that even
an approximation that does not depend on temperature or any parameter going to zero
can have a vanishing log-efficiency in the zero temperature limit!

Since this study was in part motivated by earlier heuristics on importance
sampling~\cite{grafke-vanden-eijnden:2019,ferre2020approximate}, we specifically
study these approaches. We show, under geometrical conditions, that the first and second
order techniques proposed in~\cite{grafke-vanden-eijnden:2019,ferre2020approximate} define
stochastic viscosity
approximations of order one and two respectively, thus proving variance reduction
properties for these heuristics. This also shows that our definition can indeed
be applied to actual approximation schemes.
A simple numerical example illustrates our results.

As for any new definition, connection between stochastic viscosity solutions and
existing works should be explored further. In particular, concerning variance reduction,
it seems interesting to understand better the link with the large deviation criteria
designed in~\cite{guyader2020efficient} as well as the more involved Isaacs
equations~\cite{dupuis2007subsolutions}. Moreover, it seems natural to revisit
existing approximation techniques and check whether they match our definition and, as
already noted in~\cite{leonard2021feynman}, the relation with
viscosity solutions of HJB problems should be made clearer, appart from the variance
reduction problem.

Finally,  although our work gives ground to the schemes
based on instanton expansions as the ones presented in Section~\ref{sec:instantonexpand},
such schemes
do not apply as such when the solution to the instanton problem is not unique,
which is the case in many situations. Two ways can be followed to overcome this issue.
The first is to extend the approximation to a situation with several instantons,
to build a more general abstract expansion of the HJB equation -- this is the subject of
ongoing investigation of the author. The second, more pragmatic, is
to use tools such as the region of strong regularity~\cite{fleming1992asymptotic}
to update the computation of the instanton at certain times, in the flavour
of~\cite{hairer2014improved}. Whatever approach is chosen, we believe the notion we
introduced in this work is valuable as it quantitatively shows that,
when the goal of control is variance reduction, one should focus on
approximating the HJB equation along relevant low temperature paths.

\section*{Acknowledgements}
The author particularly thanks the referee for
providing many comments that helped to significantly improve the paper.
The author is also grateful towards Charles Bertucci for reading a first
version of the manuscript and providing a series of useful comments. He also thanks
Tobias Grafke and Hugo Touchette for interesting
discussions on large deviations, as well as Djalil Chafaï and
Gabriel Stoltz for stimulating general discussions and encouragements.


\bibliographystyle{abbrv}

\end{document}